\documentclass[10pt]{amsart}

\usepackage{amssymb,amsmath,amsthm}
\usepackage[all]{xy}
\usepackage{eucal}

\usepackage{color}

\theoremstyle{plain}

\newtheorem{prop}[subsection]{Proposition}
\newtheorem{thm}[subsection]{Theorem}

\newtheorem{assumption}[subsection]{Basic Assumption}

\theoremstyle{definition}
\newtheorem{defn}[subsection]{Definition}

\theoremstyle{remark}
\newtheorem{rem}[subsection]{Remark}

\newcommand{\hh}{{ \mathsf{h} }}

\newcommand{\Ho}{{ \mathsf{Ho} }}

\newcommand{\ModR}{{ \mathsf{Mod}_\capR }}

\newcommand{\Spectra}{{ \mathsf{Sp}^\Sigma }}

\newcommand{\Alg}{{ \mathsf{Alg} }}

\newcommand{\TQ}{{ \mathsf{TQ} }}

\newcommand{\AlgO}{{ \Alg_\capO }}

\newcommand{\capO}{{ \mathcal{O} }}

\newcommand{\capR}{{ \mathcal{R} }}

\newcommand{\Nil}{{ \mathrm{Nil} }}

\newcommand{\id}{{ \mathrm{id} }}

\newcommand{\hwedge}{{ \hh\wedge }}
\newcommand{\Smash}{{ \,\wedge\, }}

\newcommand{\tensor}{{ \otimes }}

\newcommand{\wequiv}{{ \ \simeq \ }}

\newcommand{\Equal}{{ \ = \ }}

\newcommand{\function}[3]{{ {#1}\colon\thinspace{#2}\rightarrow{#3} }}

\DeclareMathOperator*{\holim}{holim}

\DeclareMathOperator{\BAR}{Bar}

\title[A nilpotent {W}hitehead theorem for $\TQ$-homology]{A nilpotent {W}hitehead theorem for $\TQ$-homology of structured ring spectra.}

\author{Michael Ching}
\author{John E. Harper}

\address{Department of Mathematics and Statistics, Amherst College, Amherst, MA, 01002, USA}

\email{mching@amherst.edu}

\address{Department of Mathematics, The Ohio State University, Newark, 1179 University Dr, Newark, OH 43055, USA}
\email{harper.903@math.osu.edu}

\begin{document}

\begin{abstract}
The aim of this short paper is to prove a $\TQ$-Whitehead theorem for nilpotent structured ring spectra. We work in the framework of symmetric spectra and algebras over operads in modules over a commutative ring spectrum. Our main result can be thought of as a $\TQ$-homology analog for structured ring spectra of Dror's generalized Whitehead theorem for topological spaces; here $\TQ$-homology is short for topological Quillen homology. We also prove retract theorems for the $\TQ$-completion and homotopy completion of nilpotent structured ring spectra.
\end{abstract}

\maketitle

\section{Introduction}

In this paper we work with algebraic structures in symmetric spectra, or $\capR$-modules, that can be described as algebras over an operad $\capO$ which is reduced, meaning that $\capO[0]=*$, the trivial $\capR$-module; such $\capO$-algebras are non-unital. Here, we assume that $\capR$ is any commutative monoid object in the category $(\Spectra,\tensor_S,S)$ of symmetric spectra \cite{Hovey_Shipley_Smith, Schwede_book_project}; in other words, $\capR$ is a commutative ring spectrum, and we denote by $(\ModR,\wedge,\capR)$ the closed symmetric monoidal category of $\capR$-modules.

Topological Quillen homology, or $\TQ$-homology for short, is the analog for $\capO$-algebras of the ordinary integral homology of spaces. There is an extensive literature, but it is useful to start with \cite{Goerss_f2_algebras, Miller, Quillen}, followed by \cite{Fresse_lie_theory, Fresse, Harper_bar_constructions, Rezk} along with \cite{Basterra, Basterra_Mandell, Basterra_Mandell_thh, Kuhn, Kuhn_adams_filtration, Kuhn_Pereira}, among others. A study of the homotopy completion tower, which is weakly equivalent to the Goodwillie Taylor tower \cite{Goodwillie_calculus_3, Kuhn_survey} of the identity functor on $\capO$-algebras (see, for instance, \cite{Harper_Hess, Kuhn,  McCarthy_Minasian_preprint, Pereira_spectral_operad}), is carried out in \cite{Harper_Hess}, where the $\TQ$-completion construction is introduced as the analog for $\capO$-algebras of the Bousfield-Kan \cite{Bousfield_Kan} integral completion of spaces; a study of $\TQ$-completion is carried out in \cite{Ching_Harper_derived_Koszul_duality, Ching_Harper} for connected $\capO$-algebras. The homotopy completion tower is exploited in \cite{Harper_Hess} to establish a $\TQ$-Serre Finiteness Theorem; also proved, using similar tools, are a $\TQ$-Hurewicz theorem and a $\TQ$-Whitehead theorem, which generalize earlier results in this direction to $\capO$-algebras. These are analogs of the Serre Finiteness, Hurewicz, and Whitehead theorems that are well-known for spaces; each of these theorems for $\TQ$-homology starts with the assumption that the $\capO$-algebra is connected. 

In this paper we explore the possibility of removing the connectivity assumption from an $\capO$-algebra, in the case of the $\TQ$-Whitehead theorem, by replacing ``connected $\capO$-algebra'' with ``nilpotent $\capO$-algebra''; this is motivated by Dror's generalized Whitehead theorem for spaces \cite{Dror_whitehead_theorem} which replaced ``simply connected space'' with ``nilpotent space'' in the original Whitehead theorem for spaces.

Our approach, which is somewhat indirect and was motivated by the work in \cite{Arone_Kankaanrinta}, is to play off the homotopy completion and $\TQ$-completion constructions against each other: more precisely, we resolve the cosimplicial $\TQ$-resolution of an $\capO$-algebra $X$, whose homotopy limit is the $\TQ$-completion of $X$, term-by-term via the homotopy completion tower, but with respect to the operad $\tau_{M-1}\capO$ characterizing $M$-nilpotent $\capO$-algebras (Remark \ref{rem:ease_of_notation}). We use this to establish a nilpotent $\TQ$-completion retract result (Theorem \ref{thm:nilpotent_TQ_completion_retract}), and as a consequence, we establish the following nilpotent $\TQ$-Whitehead theorem.

\begin{thm}[Nilpotent $\TQ$-Whitehead theorem]
\label{MainTheorem}
Let $\capO$ be an operad in $\capR$-modules with $\capO[0]=*$. Let $M\geq 2$ and suppose $X,Y$ are $M$-nilpotent $\capO$-algebras (Definition \ref{defn:nilpotent_O_algebra}). Consider any map $\function{f}{X}{Y}$ of $\capO$-algebras. Then $f$ is a weak equivalence if and only if $f$ induces a weak equivalence $\TQ(X)\wequiv\TQ(Y)$ on $\TQ$-homology. Here we are assuming that $\capO$ satisfies Assumption \ref{basic_assumption}.
\end{thm}

To keep this paper appropriately concise, we freely use notation from \cite{Harper_Hess}.

\subsection*{Acknowledgments}

The second author would like to thank Greg Arone, Bill Dwyer, and Emmanuel Farjoun for helpful remarks at an early stage of this project.  The authors would like to thank Andy Baker and Birgit Richter for the opportunity to announce these results at the  2011 Conference on Structured Ring Spectra in Hamburg, and an anonymous referee for helpful comments and suggestions. The first author was supported in part by National Science Foundation grant DMS-1308933.

\section{Nilpotent structured ring spectra}

The following basic assumption ensures, in particular, that the forgetful functor from $\capO$-algebras to the underlying category of $\capR$-modules preserves cofibrant objects; see, for instance, \cite{Pavlov_Scholbach_symmetric_powers, Pereira, Shipley_commutative_ring_spectra}.

\begin{assumption}
\label{basic_assumption}
Let $\capO$ be an operad in $\capR$-modules and consider the associated unit map $I\rightarrow\capO$. Assume that $I[r]\rightarrow\capO[r]$ is a flat stable cofibration between flat stable cofibrant objects in $\ModR$ for each $r\geq 0$ (\cite[7.7]{Harper_Hess}). 
\end{assumption}

Recall from \cite{Harper_Hess} that associated to the operad $\capO$ is the tower of operads 
\begin{align}
\label{eq:canonical_tower_of_operads}
  \tau_1\capO \leftarrow 
  \tau_2\capO \leftarrow \cdots \leftarrow
  \tau_{n-1}\capO \leftarrow 
  \tau_{n}\capO \leftarrow \cdots
\end{align}
where
\begin{align*}
  (\tau_n\capO)[t]:=
  \left\{
  \begin{array}{rl}
    \capO[t],&\text{for $t\leq n$,}\\
    *,&\text{otherwise},
  \end{array}
  \right.
\end{align*}
In other words, $\tau_n\capO$ is the operad built from $\capO$ by truncating $\capO$ above level $n$, where $*$ denotes the trivial $\capR$-module. If $X$ is an $\capO$-algebra, then the tower  \eqref{eq:canonical_tower_of_operads} of operads induces the \emph{completion tower} \cite[3.8]{Harper_Hess} of $X$
\begin{align}
\label{eq:completion_tower}
  \tau_1\capO\circ_\capO(X) \leftarrow 
  \tau_2\capO\circ_\capO(X) \leftarrow \cdots \leftarrow
  \tau_{n-1}\capO\circ_\capO(X) \leftarrow 
  \tau_{n}\capO\circ_\capO(X) \leftarrow \cdots
\end{align}
of $\capO$-algebras whose limit is the \emph{completion} $X^\wedge$ of $X$. To make this construction homotopy meaningful, one can derive everything in sight by constructing the \emph{homtopy completion tower} \cite[3.13]{Harper_Hess} of $X$
\begin{align}
\label{eq:homotopy_completion_tower}
  \tau_1\capO\circ^\hh_\capO(X) \leftarrow 
  \tau_2\capO\circ^\hh_\capO(X) \leftarrow \cdots \leftarrow
  \tau_{n-1}\capO\circ^\hh_\capO(X) \leftarrow 
  \tau_{n}\capO\circ^\hh_\capO(X) \leftarrow \cdots
\end{align}
of $\capO$-algebras whose homotopy limit is the \emph{homotopy completion} $X^\hwedge$ of $X$.

\begin{rem}
\label{rem:completion_tower}
For instance, if $\capO$ is the operad whose algebras are the non-unital commutative algebra spectra (i.e., where $\capO[t]=\capR$ for each $t\geq 1$ and $\capO[0]=*$), then the tower \eqref{eq:completion_tower} is isomorphic to the usual $X$-adic completion of $X$ tower of the form
\begin{align}
\label{eq:X_adic_completion_tower}
  X/X^2\leftarrow X/X^3\leftarrow \cdots \leftarrow X/X^{n} \leftarrow X/X^{n+1}\leftarrow \cdots
\end{align}
where $X/X^n$ denotes the $\capO$-algebra with underlying $\capR$-module defined by the pushout diagram $(n\geq 2)$
\begin{align*}
\xymatrix{
  X^{\wedge n}\ar[r]^-{m_n}\ar[d] & X\ar[d]\\
  {*}\ar[r] & X/X^n
}
\end{align*}
in $\capR$-modules; here, $m_n$ denotes the usual multiplication map.
\end{rem}

\begin{defn}
\label{defn:nilpotent_O_algebra}
Let $X$ be an $\capO$-algebra and $M\geq 2$. We say that $X$ is \emph{$M$-nilpotent} if the map of $t$-ary operations $\capO[t]\Smash X^{\wedge t}\rightarrow X$ factors as
\begin{align*}
\capO[t]\Smash X^{\wedge t}\rightarrow *\rightarrow X
\quad\quad\text{for each $t\geq M$,}
\end{align*}
where $*$ denotes the trivial $\capR$-module; in other words, if all the $M$-ary operations of $X$ and higher are trivial.
\end{defn}

\begin{rem}
Each term $X/X^n$ in \eqref{eq:X_adic_completion_tower} is an $n$-nilpotent non-unital commutative algebra spectrum, and similarly, each term $\tau_{n}\capO\circ_\capO(X)$ in the completion tower \eqref{eq:completion_tower} is an $(n+1)$-nilpotent $\capO$-algebra.
\end{rem}

\begin{rem}
\label{rem:ease_of_notation}
Since $M$-nilpotent $\capO$-algebras ($M\geq 2$) are the same as $\tau_{M-1}\capO$-algebras equipped with $\capO$-action maps (Proposition \ref{prop:truncated_O_algebras}), for ease of notational purposes (in the various subscripts) we often state assumptions in terms of $(N+1)$-nilpotent $\capO$-algebras $(N\geq 1)$.
\end{rem}

\begin{prop}
\label{prop:truncated_O_algebras}
If $X$ is an $\capO$-algebra and $N\geq 1$, then $X$ is $(N+1)$-nilpotent if and only if the left $\capO$-action map $\capO\circ(X)\rightarrow X$ factors as
\begin{align*}
  \capO\circ(X)\rightarrow (\tau_N\capO)\circ(X)\rightarrow X
\end{align*}
through a $\tau_N\capO$-algebra structure on $X$. In other words, $(N+1)$-nilpotent $\capO$-algebras are the same as $\tau_N\capO$-algebras equipped with $\capO$-action maps induced by restriction along $\capO\rightarrow\tau_N\capO$.
\end{prop}

\begin{proof}
This follows immediately from Definition \ref{defn:nilpotent_O_algebra}.
\end{proof}

\begin{rem}
It is natural to ask, if we start with a small diagram of $\tau_N\capO$-algebras, and we compute its limit (resp. colimit) in $\capO$-algebras versus in $\tau_N\capO$-algebras, when are they isomorphic? Since all small limits, filtered colimits, and reflexive coequalizers are created in the underlying category of $\capR$-modules by the forgetful functor (see, for instance, \cite{Harper_symmetric_spectra, Harper_modules_over_operads}), for both $\capO$-algebras and $\tau_N\capO$-algebras, it follows that $(N+1)$ nilpotent $\capO$-algebras are closed under all small limits, filtered colimits, and reflexive coequalizers. On the other hand, since cofibrant replacement of an $(N+1)$-nilpotent $\capO$-algebra, as an $\capO$-algebra, involves gluing on cells of $\capO$-algebras that are not nilpotent, in general, a left derived functor evaluated on a fixed $(N+1)$-nilpotent $\capO$-algebra will often take on different values depending on whether it is left derived on $\capO$-algebras versus on $\tau_N\capO$-algebras. Note that since the category of $\tau_1\capO$-algebras is isomorphic to the category of left $\capO[1]$-modules \cite{Harper_Hess}, $2$-nilpotent $\capO$-algebras are the same as left $\capO[1]$-modules equipped with a trivial $\capO$-algebra structure induced by restriction along $\capO\rightarrow\tau_1\capO$.
\end{rem}

If $N\geq 1$, repeated application of the functorial factorization construction in \cite[5.48]{Harper_Hess} shows that the upper horizontal finite tower in \eqref{eq:fattening_up_the_truncation_tower} may be factored as
\begin{align}
\label{eq:fattening_up_the_truncation_tower}
\xymatrix{
  \capO\ar[r]\ar@/_1pc/[dr]^(0.3){(*)} & 
  \tau_N\capO\ar[r] & 
  \tau_{N-1}\capO\ar[r] & 
  \cdots\ar[r] & 
  \tau_2\capO\ar[r] & 
  \tau_1\capO\\
  & J_N\ar[u]^-{(\#)}_-{\wequiv}\ar[r]^-{(*)} & 
  J_{N-1}\ar[u]^-{(\#)}_-{\wequiv}\ar[r]\ar[r]^-{(*)} & 
  \cdots\ar[r]\ar[r]^-{(*)} & 
  J_2\ar[u]^-{(\#)}_-{\wequiv}\ar[r]\ar[r]^-{(*)} & 
  J_1\ar[u]^-{(\#)}_-{\wequiv}
}
\end{align}
where the maps $(*)$ are cofibrations of operads, the maps $(\#)$ are weak equivalences of operads, and $J_n$ is reduced for each $1\leq n\leq N$.

The following is proved in \cite[5.49(b)]{Harper_Hess}.

\begin{prop}
Let $1\leq n < N$. If $X\in\Alg_{J_N}$ is cofibrant, then X is also cofibrant in $\AlgO$. Similarly, if $X\in\Alg_{J_n}$ is cofibrant, then $X$ is also cofibrant in $\Alg_{J_{n+1}}$.
\end{prop}

If $X$ is an $(N+1)$-nilpotent $\capO$-algebra, then its $\capO$-algebra structure is induced from its underlying $\tau_N\capO$-algebra structure via the composite map
\begin{align*}
  \capO\circ(X)\rightarrow\tau_N\capO\circ(X)\rightarrow X.
\end{align*}
Hence we may start with any $(N+1)$-nilpotent $\capO$-algebra X, and consider the functorial cofibrant replacement
\begin{align}
\label{eq:cofibrant_replacement_tilde_in_JN_algebras}
  *\rightarrow \tilde{X} \xrightarrow{\wequiv} X
\end{align}
of $X$ in $\Alg_{J_N}$, which is also a cofibrant replacement of $X$ in $\AlgO$; we are using the positive flat stable model structure (see, for instance, \cite{Harper_symmetric_spectra, Harper_Hess}).

\begin{defn}
If $X$ is an $(N+1)$-nilpotent $\capO$-algebra, then its \emph{$\TQ$--homology} (resp. \emph{$\TQ|_{\Nil_{N+1}}$--homology}) is the $\capO$-algebra 
\begin{align*}
  \TQ(X)&:=\tau_1\capO\circ^h_\capO(X)\wequiv J_1\circ_\capO(\tilde{X})\\
  \text{resp.}\quad\quad
  \TQ|_{\Nil_{N+1}}(X)&:=\tau_1\capO\circ^h_{\tau_N\capO}(X)\wequiv J_1\circ_{J_N}(\tilde{X})
\end{align*}
Here, ``$\TQ$--homology'' is an abbreviation for ``topological Quillen homology''. 
\end{defn}

\begin{rem}
In other words, $\TQ|_{\Nil_{N+1}}(X)$ is the topological Quillen homology of $X$ (i.e., the derived indecomposable quotient of $X$) with respect to  $\tau_N\capO$-algebras, while $\TQ(X)$ is the the derived indecomposable quotient of $X$ with respect to $\capO$-algebras; it may be helpful to note that $\tau_1\tau_N\capO=\tau_1\capO$. The indicated weak equivalences are consequences of \cite[4.10]{Harper_Hess}.
\end{rem}

This construction comes with a $\TQ$--Hurewicz map of the form $X\rightarrow\TQ(X)$, given by the map $\tilde{X}\rightarrow J_1\circ_\capO(\tilde{X})$ in $\AlgO$, and iterating this Hurewicz map leads to a cosimplicial resolution of $X$ with respect to $\TQ$-homology of the form
\begin{align*}
\xymatrix{
  X\ar[r] & 
  \TQ(X)\ar@<-0.5ex>[r]\ar@<0.5ex>[r] & 
  (\TQ)^2 (X)
  \ar@<-1.0ex>[r]\ar[r]\ar@<1.0ex>[r] &
  (\TQ)^3 (X)\cdots
}
\end{align*}
(see Remark \ref{rem:suppression_of_codegeneracy_maps}) given by the map of cosimplicial $\capO$-algebras
\begin{align}
\label{eq:point_set_TQ_resolution}
\xymatrix{
  \tilde{X}\ar[r] &
  J_1\circ_{\capO}(\tilde{X})\ar@<-0.5ex>[r]\ar@<0.5ex>[r] & 
  J_1\circ_{\capO} J_1\circ_{\capO}(\tilde{X})
  \ar@<-1.0ex>[r]\ar[r]\ar@<1.0ex>[r] &
  \cdots
}
\end{align}
where $\tilde{X}$ is regarded as a constant cosimplicial diagram. Recall \cite{Ching_Harper_derived_Koszul_duality, Harper_Hess} that applying $\holim_\Delta(-)$ to the map of cosimplicial $\capO$-algebras in \eqref{eq:point_set_TQ_resolution} gives the $\TQ$-completion map of the form $X\rightarrow X^\wedge_\TQ$.

\begin{rem}
\label{rem:suppression_of_codegeneracy_maps}
Throughout this paper we have suppressed the codegeneracy maps from the diagrams, for ease of notational purposes.
\end{rem}

Similarly, there is a $\TQ|_{\Nil_{N+1}}$--Hurewicz map of the form $X\rightarrow\TQ|_{\Nil_{N+1}}(X)$, given by the map $\tilde{X}\rightarrow J_1\circ_{J_N}(\tilde{X})$ in $\AlgO$, and iterating this Hurewicz map leads to a cosimplicial resolution of $X$ with respect to $\TQ|_{\Nil_{N+1}}$-homology of the form
\begin{align*}
\xymatrix{
  X\ar[r] & 
  \TQ|_{\Nil_{N+1}}(X)\ar@<-0.5ex>[r]\ar@<0.5ex>[r] & 
  (\TQ|_{\Nil_{N+1}})^2 (X)
  \ar@<-1.0ex>[r]\ar[r]\ar@<1.0ex>[r] &
  (\TQ|_{\Nil_{N+1}})^3 (X)\cdots
}
\end{align*}
given by the map of cosimplicial $J_N$-algebras (and by restriction along $\capO\rightarrow J_N$, also a map of cosimplicial $\capO$-algebras)
\begin{align}
\label{eq:point_set_TQ_Nil_resolution}
\xymatrix{
  \tilde{X}\ar[r] &
  J_1\circ_{J_N}(\tilde{X})\ar@<-0.5ex>[r]\ar@<0.5ex>[r] & 
  J_1\circ_{J_N} J_1\circ_{J_N}(\tilde{X})
  \ar@<-1.0ex>[r]\ar[r]\ar@<1.0ex>[r] &
  \cdots
}
\end{align}
where $\tilde{X}$ is regarded as a constant cosimplicial diagram, and applying $\holim_\Delta(-)$ to the map of cosimplicial $\capO$-algebras  gives the $\TQ|_{\Nil_{N+1}}$-completion map of the form $X\rightarrow X^\wedge_{\TQ|_{\Nil_{N+1}}}$.

The map of operads $\capO\rightarrow J_N$ induces a commutative diagram of the form
\begin{align*}
\xymatrix{
  \tilde{X}\ar[r]\ar@{=}[d] &
  J_1\circ_{\capO}(\tilde{X})\ar@<-0.5ex>[r]\ar@<0.5ex>[r]\ar[d] & 
  J_1\circ_{\capO} J_1\circ_{\capO}(\tilde{X})
  \ar@<-1.0ex>[r]\ar[r]\ar@<1.0ex>[r]\ar[d] &
  \cdots\\
  \tilde{X}\ar[r] &
  J_1\circ_{J_N}(\tilde{X})\ar@<-0.5ex>[r]\ar@<0.5ex>[r] & 
  J_1\circ_{J_N} J_1\circ_{J_N}(\tilde{X})
  \ar@<-1.0ex>[r]\ar[r]\ar@<1.0ex>[r] &
  \cdots
}
\end{align*}
and applying $\holim_\Delta(-)$ gives the left-hand commutative diagram of the form \eqref{eq:map_of_TQ_completion_maps}.

\begin{thm}[Nilpotent $\TQ$-completion retract theorem]
\label{thm:nilpotent_TQ_completion_retract}
Let $\capO$ be an operad in $\capR$-modules with $\capO[0]=*$. Let $N\geq 1$. If $X$ is an $(N+1)$-nilpotent $\capO$-algebra, then there is 
a left-hand commutative diagram of the form
\begin{align}
\label{eq:map_of_TQ_completion_maps}
\xymatrix{
  X\ar[r]\ar@{=}[d] & X^\wedge_\TQ\quad\quad\quad\ar@<-4ex>[d]\\
  X\ar[r]^-{(*)}_-{\wequiv} & X^\wedge_{\TQ|_{\Nil_{N+1}}}
}\quad\quad
\xymatrix{
\\
  X\ar[r]\ar@/^1.5pc/[rr]^-{\id} & X^\wedge_\TQ\ar[r] & X
}
\end{align}
in $\AlgO$ with natural map $(*)$ a weak equivalence. In particular, the $\TQ$-completion map $X\rightarrow X^\wedge_\TQ$ fits into the right-hand retract diagram in the homotopy category $\Ho(\AlgO)$, natural in $X$. Here we are assuming that $\capO$ satisfies Assumption \ref{basic_assumption}.
\end{thm}

The purpose of the rest of this section is to prove Theorems \ref{thm:nilpotent_TQ_completion_retract} and \ref{MainTheorem}. Let's verify that the comparison map $(*)$ is a weak equivalence. Our approach is to resolve the coaugmented cosimplicial diagram \eqref{eq:point_set_TQ_Nil_resolution} of $J_N$-algebras term-by-term with the finite tower $\{J_n\circ_{J_N}(-)\}_n$, where $1\leq n\leq N$; we are motivated by the work in \cite{Arone_Kankaanrinta}. This leads to the commutative diagram in $\AlgO$ of the form
\begin{align*}
\xymatrix{
  \tilde{X}\ar[r]^-{(\#)}\ar@{=}[d] &
  J_1\circ_{J_N}(\tilde{X})\ar@<-0.5ex>[r]\ar@<0.5ex>[r]\ar@{=}[d] & 
  J_1\circ_{J_N} J_1\circ_{J_N}(\tilde{X})
  \ar@<-1.0ex>[r]\ar[r]\ar@<1.0ex>[r]\ar@{=}[d] &
  \cdots\\
  J_N\circ_{J_N}(\tilde{X})\ar[r]^-{(\#)_N}\ar[d] &
  J_N\circ_{J_N}\bigl(J_1\circ_{J_N}(\tilde{X})\bigr)\ar@<-0.5ex>[r]\ar@<0.5ex>[r]\ar[d] & 
  J_N\circ_{J_N}\bigl(J_1\circ_{J_N} J_1\circ_{J_N}(\tilde{X})\bigr)
  \ar@<-1.0ex>[r]\ar[r]\ar@<1.0ex>[r]\ar[d] &
  \cdots\\
  \vdots\ar[d] &
  \vdots\ar[d] & 
  \vdots\ar[d] &
  \cdots\\
  J_2\circ_{J_N}(\tilde{X})\ar[r]^-{(\#)_2}\ar[d] &
  J_2\circ_{J_N}\bigl(J_1\circ_{J_N}(\tilde{X})\bigr)\ar@<-0.5ex>[r]\ar@<0.5ex>[r]\ar[d] & 
  J_2\circ_{J_N}\bigl(J_1\circ_{J_N} J_1\circ_{J_N}(\tilde{X})\bigr)
  \ar@<-1.0ex>[r]\ar[r]\ar@<1.0ex>[r]\ar[d] &
  \cdots\\
  J_1\circ_{J_N}(\tilde{X})\ar[r]^-{(\#)_1} &
  J_1\circ_{J_N}\bigl(J_1\circ_{J_N}(\tilde{X})\bigr)\ar@<-0.5ex>[r]\ar@<0.5ex>[r] & 
  J_1\circ_{J_N}\bigl(J_1\circ_{J_N} J_1\circ_{J_N}(\tilde{X})\bigr)
  \ar@<-1.0ex>[r]\ar[r]\ar@<1.0ex>[r] &
  \cdots
}
\end{align*}

We will show by induction up the finite tower that $\holim_\Delta(\#)_n$ is a weak equivalence for each $1\leq n\leq N$. Since the bottom horizontal coaugmented cosimplicial diagram has extra codegeneracy maps $s^{-1}$ (\cite[6.2]{Dwyer_Miller_Neisendorfer}) induced by the map $J_N\rightarrow J_1$, by cofinality \cite[3.16]{Dror_Dwyer_long_homology} it follows that $\holim_\Delta(\#)_1$ is a weak equivalence. To make further progress, we must identify the homotopy fibers of these towers.

\begin{prop}
Let $Z$ be a cofibrant $J_N$-algebra. There is a zigzag of weak equivalences of towers of the form
\begin{align*}
  \{\tau_nJ_N\circ_{J_N}(Z)\}_n\wequiv
  \{J_n\circ_{J_N}(Z)\}_n
\end{align*}
natural with respect to all such $Z$; here, $1\leq n\leq N$.
\end{prop}

\begin{proof}
There is a zigzag of weak equivalences of the form
\begin{align*}
  \tau_nJ_N\circ_{J_N}(Z)\xleftarrow{\wequiv}
  |\BAR(\tau_nJ_N,J_N,Z)|\xrightarrow{\wequiv}\\
  |\BAR(\tau_n\capO,J_N,Z)|\xleftarrow{\wequiv}
  |\BAR(J_n,J_N,Z)|\xrightarrow{\wequiv}J_n\circ_{J_N}(Z)
\end{align*}
natural in all such $Z$; we are freely using the ``$\BAR$'' notation from \cite[4.10]{Harper_Hess}.
\end{proof}

\begin{prop}
\label{prop:homotopy_fiber_sequence}
Let $Z$ be a cofibrant $J_N$-algebra. For each $2\leq n\leq N$, there is a homotopy fiber sequence of the form
\begin{align*}
  F_n(Z)\rightarrow J_n\circ_{J_N}(Z) \rightarrow J_{n-1}\circ_{J_N}(Z)
\end{align*}
in $\AlgO$, natural in all such $Z$, where $F_n(Z)$ is defined by
\begin{align}
\label{eq:homotopy_fiber}
  F_n(Z):= i_nJ_N\circ^\hh_{\tau_1J_N}\bigl(\tau_1J_1\circ_{J_1}J_1\circ_{J_N}(Z)\bigr)
\end{align}
\end{prop}

\begin{proof}
By Proposition \ref{prop:homotopy_fiber_sequence}, together with \cite[4.21]{Harper_Hess}, there is a homotopy fiber sequence of the form
\begin{align*}
  i_nJ_N\circ^\hh_{\tau_1J_N}\bigl(\tau_1J_N\circ_{J_N}(Z)\bigr)
  \rightarrow J_n\circ_{J_N}(Z) \rightarrow J_{n-1}\circ_{J_N}(Z)
\end{align*}
in $\AlgO$, natural in all such $Z$. Noting that the weak equivalence $\tau_1J_N\xrightarrow{\wequiv}\tau_1J_1$ induces a weak equivalence of the form
\begin{align*}
  i_nJ_N\circ^\hh_{\tau_1J_N}\bigl(\tau_1J_N\circ_{J_N}(Z)\bigr)
  &\wequiv
  i_nJ_N\circ^\hh_{\tau_1J_N}\bigl(\tau_1J_1\circ_{J_N}(Z)\bigr)\\
  &\wequiv
  i_nJ_N\circ^\hh_{\tau_1J_N}\bigl(\tau_1J_1\circ_{J_1}J_1\circ_{J_N}(Z)\bigr)
  \Equal F_n(Z)
\end{align*}
completes the proof.
\end{proof}

\begin{proof}[Proof of Theorem \ref{thm:nilpotent_TQ_completion_retract}]

Consider the commutative diagram of the form
\begin{align}
\label{eq:map_of_cosimplicial_objects_useful}
\xymatrix{
  F_n(\tilde{X})\ar[r]^-{(*)_n}\ar[d] &
  F_n\bigl(J_1\circ_{J_N}(\tilde{X})\bigr)\ar@<-0.5ex>[r]\ar@<0.5ex>[r]\ar[d] & 
  F_n\bigl(J_1\circ_{J_N} J_1\circ_{J_N}(\tilde{X})\bigr)
  \ar@<-1.0ex>[r]\ar[r]\ar@<1.0ex>[r]\ar[d] &
  \cdots\\
  J_n\circ_{J_N}(\tilde{X})\ar[r]^-{(\#)_n}\ar[d] &
  J_n\circ_{J_N}\bigl(J_1\circ_{J_N}(\tilde{X})\bigr)\ar@<-0.5ex>[r]\ar@<0.5ex>[r]\ar[d] & 
  J_n\circ_{J_N}\bigl(J_1\circ_{J_N} J_1\circ_{J_N}(\tilde{X})\bigr)
  \ar@<-1.0ex>[r]\ar[r]\ar@<1.0ex>[r]\ar[d] &
  \cdots\\
  J_{n-1}\circ_{J_N}(\tilde{X})\ar[r]^-{(\#)_{n-1}} &
  J_{n-1}\circ_{J_N}\bigl(J_1\circ_{J_N}(\tilde{X})\bigr)\ar@<-0.5ex>[r]\ar@<0.5ex>[r] & 
  J_{n-1}\circ_{J_N}\bigl(J_1\circ_{J_N} J_1\circ_{J_N}(\tilde{X})\bigr)
  \ar@<-1.0ex>[r]\ar[r]\ar@<1.0ex>[r] &
  \cdots
}
\end{align}
in $\AlgO$, for each $(2\leq n\leq N)$. We know that the coaugmented cosimplicial diagram in the top horizontal row has extra codegeneracy maps $s^{-1}$ (\cite[6.2]{Dwyer_Miller_Neisendorfer}), since it has the form of a functor $i_nJ_N\circ^\hh_{\tau_1J_N}\bigl(\tau_1J_1\circ_{J_1}-\bigr)$ applied to the coaugmented cosimplicial diagram of the form
\begin{align*}
\xymatrix{
J_{1}\circ_{J_N}(\tilde{X})\ar[r]^-{(\#)_{1}} &
  J_{1}\circ_{J_N}\bigl(J_1\circ_{J_N}(\tilde{X})\bigr)\ar@<-0.5ex>[r]\ar@<0.5ex>[r] & 
  J_{1}\circ_{J_N}\bigl(J_1\circ_{J_N} J_1\circ_{J_N}(\tilde{X})\bigr)
  \ar@<-1.0ex>[r]\ar[r]\ar@<1.0ex>[r] &
  \cdots
}
\end{align*}
in $\Alg_{J_1}$ which has extra codegeneracy maps $s^{-1}$ (\cite[6.2]{Dwyer_Miller_Neisendorfer}) in $\Alg_{J_1}$ induced by the map $J_N\rightarrow J_1$. It follows from cofinality \cite[3.16]{Dror_Dwyer_long_homology} that $\holim_\Delta(\#)_1$ and $\holim_\Delta(*)_n$ are weak equivalences for each $2\leq n\leq N$. Applying $\holim_\Delta$ to \eqref{eq:map_of_cosimplicial_objects_useful}, regarded as a commutative diagram of $\Delta$-shaped diagrams in $\AlgO$,  we have a commutative diagram of the form
\begin{align}
\label{eq:map_of_homotopy_fiber_sequences_finite_tower}
\xymatrix{
  F_n(\tilde{X})\ar[rr]^-{\holim_\Delta(*)_n}_-{\wequiv}\ar[d] && 
  \holim_\Delta\Bigl(F_n\bigl(J_1\circ_{J_N}(\tilde{X})\bigr)\Rightarrow\cdots\Bigr)\ar@<-13.0ex>[d]\\
  J_n\circ_{J_N}(\tilde{X})\ar[rr]^-{\holim_\Delta(\#)_n}\ar[d] && 
  \holim_\Delta\Bigl(J_n\circ_{J_N}\bigl(J_1\circ_{J_N}(\tilde{X})\bigr)\Rightarrow \cdots \Bigr)\ar@<-13.0ex>[d]\\
  J_{n-1}\circ_{J_N}(\tilde{X})\ar[rr]^-{\holim_\Delta(\#)_{n-1}} && 
  \holim_\Delta\Bigl(J_{n-1}\circ_{J_N}\bigl(J_1\circ_{J_N}(\tilde{X})\bigr)\Rightarrow\cdots\Bigr)
}
\end{align}
in $\AlgO$, with vertical columns homotopy fiber sequences, since homotopy limits commute, up to weak equivalence. Consider diagram \eqref{eq:map_of_homotopy_fiber_sequences_finite_tower} in the case $n=2$; then since the top and bottom horizontal maps are weak equivalences, it follows that the middle horizontal map $\holim_\Delta(\#)_2$ is a weak equivalence; we are using the fact that every homotopy fiber sequence has an associated long exact sequence in (derived or true) homotopy groups of spectra, together with the five lemma from homological algebra (see, for instance, \cite{Schwede_book_project, Schwede_homotopy_groups}). Hence, by induction up the tower, it follows that $\holim_\Delta(\#)_n$ is a weak equivalence for each $1\leq n\leq N$. In particular, $\holim_\Delta(\#)=\holim_\Delta(\#)_N$ is a weak equivalence, but this is precisely the map
$X\xrightarrow{\wequiv}X^\wedge_{\TQ|_{\Nil_{N+1}}}$
in diagram \eqref{eq:map_of_TQ_completion_maps}, which completes the proof.
\end{proof}

\begin{proof}[Proof of Theorem \ref{MainTheorem}]
The ``only if'' direction is trivial since $\TQ$-homology preserves weak equivalences by construction. Conversely, suppose $f$ induces a weak equivalence $\TQ(X)\wequiv\TQ(Y)$ on $\TQ$-homology. Consider the commutative diagram of the form (by Theorem \ref{thm:nilpotent_TQ_completion_retract})
\begin{align}
\label{eq:retraction_of_completion_map}
\xymatrix{
  X\ar[r]\ar[d]_-{f} & 
  X^\wedge_\TQ\ar[r]\ar[d]^-{(\#)}_-{\wequiv} & 
  X^\wedge_{\TQ|_{\Nil_{M}}}\ar[d] & 
  X\ar[l]_-{\wequiv}\ar[d]^-{f}\\
  Y\ar[r] & 
  Y^\wedge_\TQ\ar[r] & 
  Y^\wedge_{\TQ|_{\Nil_{M}}} & 
  Y\ar[l]_-{\wequiv}
}
\end{align}
Since $\TQ(X)\rightarrow\TQ(Y)$ is a weak equivalence, it follows that $(\#)$ is a weak equivalence, and hence $f$ is a weak equivalence; this is because diagram \eqref{eq:retraction_of_completion_map} is a retract diagram by \eqref{eq:map_of_TQ_completion_maps} after passing to the homotopy category $\Ho(\AlgO)$.
\end{proof}

The purpose of the rest of this section is to point out, for completeness, that a similar retract property is also true with $\TQ$-completion replaced by homotopy completion.

\begin{prop}
Let $Z$ be a $J_N$-algebra. There are maps of towers of the form
\begin{align*}
  \{\tau_n\capO\circ_{\capO}(Z)\}_n\rightarrow
  \{\tau_nJ_N\circ_{\capO}(Z)\}_n\rightarrow
  \{\tau_nJ_N\circ_{J_N}(Z)\}_n
\end{align*}
natural with respect to $Z$.
\end{prop}

\begin{proof}
This is because the left-hand map is induced by $\tau_n\capO\rightarrow\tau_nJ_N$ and the right-hand map is induced by $\capO\rightarrow J_N$.
\end{proof}

It follows that there is a commutative diagram of towers
\begin{align*}
\xymatrix{
  \{\tilde{X}\}_n\ar[r]\ar@{=}[d] & 
  \{\tau_n\capO\circ_{\capO}(\tilde{X})\}_n\ar[d]\\
  \{\tilde{X}\}_n\ar[r] & 
  \{\tau_nJ_N\circ_{J_N}(\tilde{X})\}_n
}
\end{align*}
of $\capO$-algebras, where $\{\tilde{X}\}_n$ denotes the constant tower with value $\tilde{X}$, and applying $\holim_n$ gives the left-hand commutative diagram of the form \eqref{eq:map_of_homotopy_completion_maps}.

\begin{thm}[Nilpotent homotopy completion retract theorem]
Let $\capO$ be an operad in $\capR$-modules with $\capO[0]=*$. Let $N\geq 1$. If $X$ is an $(N+1)$-nilpotent $\capO$-algebra, then there is 
a left-hand commutative diagram of the form
\begin{align}
\label{eq:map_of_homotopy_completion_maps}
\xymatrix{
  X\ar[r]\ar@{=}[d] & X^\hwedge\quad\quad\quad\ar@<-4ex>[d]\\
  X\ar[r]^-{(*)}_-{\wequiv} & X^\hwedge|_{\Nil_{N+1}}
}\quad\quad
\xymatrix{
\\
  X\ar[r]\ar@/^1.5pc/[rr]^-{\id} & X^\hwedge\ar[r] & X
}
\end{align}
in $\AlgO$ 
with natural map $(*)$ a weak equivalence. In particular, the homotopy completion map $X\rightarrow X^\hwedge$ fits into the right-hand retract diagram in the homotopy category $\Ho(\AlgO)$, natural in $X$. Here we are assuming that $\capO$ satisfies Assumption \ref{basic_assumption}.
\end{thm}

\begin{proof}
It suffices to observe that $(*)$ is a weak equivalence. This follows from the fact that $\tilde{X}\xrightarrow{\wequiv}\tau_nJ_N\circ_{J_N}(\tilde{X})$ is a weak equivalence for each $n\geq N$; this follows from the observation that the map $J_N\xrightarrow{\wequiv}\tau_nJ_N$ is a weak equivalence for each $n\geq N$.
\end{proof}

\bibliographystyle{plain}
\bibliography{NilpotentStructured}

\end{document}